\documentclass[preprint,12pt]{elsarticle}
\usepackage{}

\usepackage{amsfonts}
\usepackage{amsmath}
\usepackage{amssymb}
\usepackage{graphics}
\usepackage{graphicx}
\usepackage{mathrsfs}
\usepackage{indentfirst}
\usepackage{stmaryrd}
\usepackage{latexsym}
\usepackage{xypic}
\newtheorem{definition}{Definition}[section]

\newtheorem{theorem}[definition]{Theorem}
\newtheorem{proposition}[definition]{Proposition}
\newtheorem{example}[definition]{Example}

\newtheorem{remark}[definition]{Remark}

\newproof{proof}{\textbf{Proof}}

\journal{Indagationes Mathematicae}

\begin{document}

\begin{frontmatter}



\title{\textbf{Algebraic results on universal quantifiers in monoidal t-norm based logic}}


\author{Juntao Wang$^{*}$}
\cortext[cor1]{Corresponding author. \\
Email addresses: wjt@xsyu.edu.cn (J.T.Wang).}

\address {School of Science, Xi'an Shiyou University, Xi'an 710065, Shaanxi, China}

\begin{abstract}
In this paper, we enlarge the language of MTL-algebras by a unary operation $\forall$ equationally described so as to abstract algebraic properties of the universal quantifier ``for any" in its original meaning. The resulting class of algebras will be called \emph{MTL-algebras with universal quantifiers} (UMTL-algebras for short). After discussing some basic algebraic properties of UMTL-algebras, we start a systematic study of the main subclasses of UMTL-algebras, some of which constitute well known algebras: UMV-algebras and monadic Boolean algebra. Then we give some characterizations of representable, simple, semsimple UMTL-algebras, and obtain some representations of UMTL-algebras. Finally, we establish  modal monoidal t-norm based logic  and prove that is completeness with respect to the variety of UMTL-algebras,  and then obtain that a necessary and sufficient condition for the modal monoidal t-norm based logic to be semilinear.

\end{abstract}


\end{frontmatter}

\section{Introduction}
\label{intro}
Non-classical logic takes the advantage of the classical logic to handle
uncertain information and fuzzy information. In recent decades, various logical algebras have been proposed as the semantical systems of
non classical logic, for example, MV-algebras, BL-algebras,
G\"{o}del algebras and MTL-algebras. Among these logical algebras, MTL-algebras are the most significant structures, since the others are
all particular
cases of them. MTL-algebras are the corresponding algebraic structures of monoidal t-norm based logic $\mathbf{MTL}$, which was introduced in \cite{Esteva} in order to give the propositional logic corresponding to left-continuous t-norms and their residua. As an algebraic point of view, MTL-algebras contain all algebras
induced by left continuous t-norm and their residua \cite{Jenei}. $\mathbf{MTL}$ can also be seen as a weaker logic of the H\'{a}jek's Basic Logic $\mathbf{BL}$, a general framework in which tautologies of continuous t-norm and their residua can be captured, by dropping this divisibility condition \cite{Hajek}.

The notion of a quantifier on a Boolean algebra was introduced by Halmos in \cite{Halmos} as an algebraic counterpart of the logical notion of an existential quantifier, and the algebras obtained in this way were called by Halmos monadic Boolean algebras. After then quantifiers have been considered by several authors in different algebras, for example, orthomodular lattices, Heyting algebras, distributive lattices, MV-algebras (Wajsberg algebras), BL-algebras, NM-algebras and BCI-algebras\cite{Janowitz,Monteiro1,Cignoli,Lattanzi,Nola,Nola1,Daniel,Xin,Xin1}. In the above-mentioned algebras, both MV-algebras and NM-algebras satisfy De Morgan and double negation laws, in the definition of the corresponding quantifiers, it is possible to use only one of the existential and universal quantifiers as primitive, the other being definable as the dual of the one defined. However, definitions of quantifiers on distributive lattice, BL-algebras and MTL-algebras are not mutually interdefinable. So, which quantifier (universal or existential) as the initial one introduce in MTL-algebras is very important. In particular, in order to study monadic ideal and related theory, the existential quantifier instead of the universal quantifier were used to define monadic Boolean algebra. But for MTL-algebra, we would like to study filter rather than ideal, using the universal quantifier as the original one is more natural and convenient. Thus, we introduce and study universal quantifier on MTL-algebra in this paper.

In this paper, we give an algebraic study of the universal quantifiers in monoidal t-norm based logic. More precisely, we first introduce the concept of MTL-algebras with universal quantifiers and show that is a natural generalization of MV-algebras with universal quantifiers. Then we characterize the classes of MTL-algebras with universal quantifiers and give some representations of them. Finally, we establish modal monoidal t-norm based propositional logics and prove that is completeness with respect to the variety of MTL-algebras with universal quantifiers, and then obtain a necessary and sufficient condition for the modal monoidal t-norm based logic to be semilinear.

The paper is organized as follows: In Section 2, we review some results on MTL-algebras and their corresponding logics. In Section 3, we introduce the concept of MTL-algebras with universal quantifier and study some related properties of them. In Section 4, we give some characterizations of subclasses of MTL-algebras with universal quantifiers and give some representations of them. In Section 5, we introduce the modal monoidal t-norm based propositional logic and prove that is completeness with respect to the variety of MTL-algebras with quantifiers.

\section{Preliminaries}

In this section, we summarize some definitions and results about MTL-algebras and their corresponding logics, which will be used in the following sections.

\begin{definition}\emph{\cite{Esteva} (Monoidal t-norm based logic) $\mathbf{MTL}$ is the
logic given by the Hilbert-style calculus with $\mathbf{MP}$ as its only inference rule and the following axioms:}
\begin{enumerate}[(1)]
  \item \emph{$(\alpha\Rightarrow \beta)\Rightarrow((\beta\Rightarrow\gamma)\Rightarrow(\alpha\Rightarrow\gamma))$,}
  \item \emph{$(\alpha\&\beta)\Rightarrow \alpha$,}
  \item \emph{$(\alpha\&\beta)\Rightarrow(\beta\&\alpha)$,}
  \item \emph{$(\alpha\sqcap \beta)\Rightarrow \alpha$,}
  \item \emph{$(\alpha\sqcap \beta)\Rightarrow (\beta\sqcap \alpha)$,}
  \item \emph{$(\alpha\&(\alpha\Rightarrow \beta)\Rightarrow(\alpha\sqcap\beta)$,}
  \item \emph{$(\alpha\Rightarrow(\beta\Rightarrow\gamma))\Rightarrow((\alpha\&\beta)\Rightarrow\gamma)$,}
  \item \emph{$((\alpha\&\beta)\Rightarrow\gamma)\Rightarrow(\alpha\Rightarrow(\beta\Rightarrow \gamma))$,}
  \item \emph{$((\alpha\Rightarrow\beta)\Rightarrow\gamma)\Rightarrow(((\beta\Rightarrow \alpha)\Rightarrow \gamma)\Rightarrow \gamma)$,}
  \item \emph{ \={0}$\Rightarrow\alpha $,}
\end{enumerate}
\end{definition}

Other connectives in $\mathbf{MTL}$ can be defined from $\&,\Rightarrow,\sqcap$ as follows:
\begin{center}
$\overline{1}=\alpha\Rightarrow\alpha, \neg \alpha=\alpha\Rightarrow \overline{0}$,

$\alpha\sqcup \beta=((\alpha\Rightarrow \beta)\Rightarrow \beta)\sqcap (\beta\Rightarrow \alpha)\Rightarrow \alpha)$,

$\alpha\Leftrightarrow \beta=(\alpha\Rightarrow \beta)\&(\beta\Rightarrow \alpha)$.
\end{center}

Involutivity  monoidal t-norm based logic $\mathbf{IMTL}$ is obtained from $\mathbf{MTL}$ by adding the following axiom:
\begin{center}
  (INV)  $\neg\neg\alpha\Rightarrow\alpha$.
\end{center}

Nilpotent Minimum Logic $\mathbf{NM}$ is obtained from $\mathbf{IMTL}$ by adding the following axiom:
\begin{center}
(WNM)   $\neg(\alpha\&\beta)\sqcup ((\alpha\sqcap \beta)\Rightarrow (\alpha\&\beta))$.
\end{center}

{\L}ukasiewicz logic $\mathbf{MV}$can be obtained by adding to $\mathbf{MTL}$ the following axiom:
\begin{center}
  (MV)  $((\alpha\Rightarrow \beta)\Rightarrow \beta)\Rightarrow ((\beta\Rightarrow \alpha)\Rightarrow \alpha)$.
\end{center}

Classical propositional logic $\mathbf{CL}$ can be obtained by adding to $\mathbf{MTL}$ the excluded middle axiom:

\begin{center}
  (EM)  $\alpha\sqcup \neg \alpha$.
\end{center}

As pointed out in \cite{Noguera6}, $\mathbf{MTL}$ and its axiomatic extensions are all algebraizable in the sense of Blok and Pigozzi \cite{Blok}, and their corresponding semantics forms a variety of algebras. The variety $\mathbb{MTL}$ of MTL-algebras coincides with the variety of prelinear residuated lattices. The algebras of this variety are subdirect products of the linearly ordered members of the class. This also gives completeness of $\mathbf{MTL}$ with respect to the class of linearly ordered MTL-algebras \cite{Esteva}.

\begin{definition}\emph{\cite{Esteva} An algebraic structure $(L,\wedge,\vee,\odot,\rightarrow,0,1)$ of type $(2,2,2,\\2,0,0)$ is called an \emph{MTL-algebra} if it satisfies the following conditions:}
\begin{enumerate}[(1)]
  \item \emph{$(L,\wedge,\vee,0,1)$ is a bounded lattice,}
  \item \emph{$(L,\odot,1)$ is a commutative monoid,}
  \item \emph{$x\odot y\leq z$ if and only if $x\leq y\rightarrow z$,}
  \item \emph{$(x\rightarrow y)\vee(y\rightarrow x)=1$, for any $x,y,z\in L$.}
\end{enumerate}
\end{definition}

An MTL-algebra is called an \emph{IMTL-algebra} if it satisfies the involutivity equation:
\begin{center} (INV) $\neg\neg x=x$.
\end{center}

An IMTL-algebra is called an \emph{NM-algebra} if it satisfies the additional equation:
\begin{center} (WNM) $(x\odot y\rightarrow 0)\vee(x\wedge y\rightarrow x\odot y)=1$.
\end{center}

An MTL-algebra is called an \emph{MV-algebra} if it satisfies the following equation:
\begin{center} (MV) $(x\rightarrow y)\rightarrow y=(y\rightarrow x)\rightarrow x$.
\end{center}

An MTL-algebra is called a \emph{Boolean algebra} if it satisfies the excluded middle equation:
\begin{center} (EM) $x\vee \neg x=1$.
\end{center}

 Let $L$ be an MTL-algebra. A nonempty subset $F$ of $L$ is called a \emph{filter} if  it satisfies: (1) $1\in F$; (2) $x\in F$ and $x\rightarrow y\in F$ imply $y\in F$.  A filter $F$ of $L$ is called a \emph{proper filter} if $F\neq L$. Unless otherwise explicitly stated, filters are assumed to be proper. A proper filter $F$ of $L$ is called a \emph{maximal filter} if it is not contained in any proper filter of $L$. A proper filter $F$ of $L$ is called a \emph{prime filter} if for each $x,y\in L$ and $x\vee y\in F$, imply $x\in F$ or $y\in F$. A prime filter $F$ is said to be \emph{minimal} if it is a minimal element in the set of prime filters of $L$ ordered by inclusion. Moreover, we denote by $\langle X\rangle$ is the filter generated by a nonempty subset $X$ of $L$. Clearly, we have $\langle X\rangle=\{x\in L|x\geq x_1\odot x_2\odot\cdots\odot x_n$, for some $n\in N$ and some $x_i\in X\}$. In particular, the principal filter generated by an element $x\in L$ is $\langle x\rangle=\{y\in L|y\geq x^n \}$. If $F$ is a filter and $x\in L$, then $\langle F\cup x\rangle=\{y\in L|y \geq f\odot x^n$, for some $f\in F\}$. We denote by $F[L]$ be the set of all filers of $L$ and obtain that  $(F[L],\subseteq )$ forms a complete lattice \cite{Esteva,Wang,Wang1}.\medskip

 \begin{theorem}\emph{\cite{Zhang2}} \emph{ Let $L$ be an MTL-algebra and $P$ be a proper filter of $L$. Then the following statements are equivalent:
\begin{enumerate}[(1)]
  \item $P$ is a minimal prime,
  \item $P=\cup\{a^\bot|a\notin P\}$, where $a^\bot=\{x\in L|a\vee x=1\}$.
\end{enumerate}}
\end{theorem}

\begin{definition}\emph{\cite{Zhang2} An MTL-algebra $L$ is called \emph{representable} if $L$ is isomorphic to a subdirect product
of linearly ordered MTL-algebras.}
\end{definition}

\begin{theorem}\emph{\cite{Zhang2} Let $L$ be an MTL-algebra. Then the  following statements are equivalent:}
\begin{enumerate}[(1)]
  \item \emph{$L$ is a representable MTL-algebra,}
  \item \emph{there exists a set $\mathcal{P}$ of prime filters such that $\cap \mathcal{P}=\{1\}$.}
\end{enumerate}
\end{theorem}

\section{MTL-algebras with universal quantifiers}

In this section, we introduce MTL-algebras with universal quantifiers and investigate some related of their properties. Then we show that MTL-algebras with universal quantifiers are natural generalize MV-algebras with universal quantifiers.

\begin{definition}\emph{An MTL-algebra with a universal quantifier (\emph{UMTL-algebra} for short) is a structure $(L,\forall)=(L,\wedge,\vee,\odot,\rightarrow,0,1,\forall)$, where $(L,\wedge,\vee,\odot,\rightarrow,0,1)$ is an MTL-algebra and $\forall:L\rightarrow L$ is an unary operation on $L$ satisfying, for any $x,y\in L$,}
\begin{enumerate}[(U1)]
  \item $\forall(x)\leq x$,
  \item $\forall(x\rightarrow \forall y\rightarrow \forall y)=(\forall x\rightarrow \forall y)\rightarrow \forall y$,
  \item $\forall(\forall x\rightarrow y)=\forall x\rightarrow\forall y$.
\end{enumerate}
\end{definition}

Clearly the class of UMTL-algebras constitutes a variety which will be henceforth denoted by $\mathbb{UMTL}$.

\begin{example}
 \emph{Let $L=\{0,a,b,c,d,1\}$ be a set such that $0\leq a,b; a\leq c,d; b\leq c; c,d\leq 1$. Defining operations $\odot$ and $\rightarrow$ as follows:}
 \begin{center}
 \begin{tabular}{c|c c c c c c}
   $\odot$ & $0$ & $a$ & $b$ & $c$ & $d$ & $1$\\
   \hline
   $0$ & $0$ & $0$ & $0$ & $0$ & $0$ & $0$ \\
   $a$ & $0$ & $0$ & $0$ & $0$ & $0$ & $a$ \\
   $b$ & $0$ & $0$ & $b$ & $b$ & $0$ & $b$ \\
   $c$ & $0$ & $0$ & $b$ & $b$ & $a$ & $c$\\
   $d$ & $0$ & $0$ & $0$ & $0$ & $d$ & $d
   $ \\
   $1$ & $0$ & $a$ & $b$ & $c$ & $d$ & $1$
 \end{tabular} {\qquad}
 \begin{tabular}{c|c c c c c c}
    $\rightarrow$ & $0$ & $a$ & $b$ & $c$ & $d$ & $1$\\
    \hline
    $0$ & $1$ & $1$ & $1$ & $1$ & $1$ & $1$ \\
    $a$ & $c$ & $1$ & $c$ & $1$ & $1$ & $1$ \\
    $b$ & $d$ & $d$ & $1$ & $1$ & $d$ & $1$ \\
    $c$ & $a$ & $d$ & $c$ & $1$ & $d$ & $1$\\
    $d$ & $b$ & $c$ & $b$ & $c$ & $1$ & $1$ \\
    $1$ & $0$ & $
   a$ & $b$ & $c$ & $d$ & $1$
  \end{tabular}
  \end{center}
\emph{Then  $(L,\wedge,\vee,\odot,\rightarrow,0,1)$ is an MTL-algebra. Now, we define $\forall$ as follow:
\begin{center}
                  $\forall x=
                 \begin{cases} 1, & x=1\\
                 0, & x\neq 1
                  \end{cases}$,
                   \end{center}
It is easily verified that $(L,\forall)$ is a UMTL-algebra.}
\end{example}

\begin{example}\emph{ Let $L$ be a standard NM-algebra on $[0,1]$ and $L_n\subseteq L$ be a standard $n$-valued NM-algebra for some $n\geq 2$ (its elements are $0,\frac{1}{n-1},\cdots,\frac{n-2}{n-1},1)$. For any $x\in L$, we define
\begin{center}
$\forall x=\max\{y\in L_n|y\leq x\}$.
\end{center}
 Then $(L,\forall)$ is a UMTL-algebra.}
\end{example}

\begin{proposition}\emph{ In any UMTL-algebra $(L,\forall)$ the following properties hold:}
\begin{enumerate}[(1)]
  \item \emph{$\forall 0=0$,
  \item $\forall 1=1$,
  \item $\forall\forall x=\forall x$,
  \item $x\leq y$ implies $\forall x\leq \forall y$,
  \item $\forall(x\rightarrow y)\leq \forall x\rightarrow\forall y$, especially, $\forall \neg x\leq \neg \forall x$,
  \item $\forall x\leq y$ if and only if $\forall x\leq \forall y$,
  \item $\forall(\forall x\rightarrow \forall y)=\forall x\rightarrow \forall y$,
  \item $\forall\neg \forall x=\neg \forall x$,
  \item $\forall(x\wedge y)=\forall x\wedge \forall y$,
  \item $\forall(\forall x\vee \forall y)=\forall x\vee \forall y$,
  \item $\forall(x\odot y)\geq\forall x\odot\forall y$,
  \item $\forall(\forall x\odot \forall y)=\forall x\odot \forall y$,
  \item $\forall L=L_{\forall}$, where $L_{\forall}=\{x\in L|\forall x=x\}$,
  \item $\forall L$ is a subalgebra of $L$.}
\end{enumerate}
\end{proposition}
\begin{proof}\begin{enumerate}[(1)]
               \item Applying (U1), we have $\forall 0 \leq 0$. Thus, $\forall 0=0$.
               \item From (U1) and (U3), we have
               \begin{center}
               $\forall 1=\forall(\forall x\rightarrow x)=\forall x\rightarrow \forall x=1$.
               \end{center}
               \item From (U2) and (1), we have
               \begin{center}
                $\forall\forall x=\forall(0\rightarrow\forall x\rightarrow \forall x)=(\forall 0\rightarrow \forall x)\rightarrow \forall x=1\rightarrow \forall x=\forall x$.
                \end{center}
               \item If $x\leq y$, then $x\rightarrow y=1$. It follows from (U3) and (2) that
               \begin{center}
               $1=\forall (1)=\forall(\forall x\rightarrow y)=\forall x\rightarrow \forall y$.
               \end{center}
               which implies that $\forall x\leq \forall y$.
               \item From (U1), we get $x\rightarrow y\leq \forall x\rightarrow y$, and hence by (Q3) and (4), we have
               \begin{center}
               $\forall(x\rightarrow y)\leq \forall(\forall x\rightarrow y)=\forall x\rightarrow \forall y$.
               \end{center}
               \item Clearly.
               \item From (U3) and (3), we deduce that
               \begin{center}
               $\forall(\forall x\rightarrow \forall y)=\forall x\rightarrow \forall\forall y=\forall x\rightarrow \forall y$.
               \end{center}
               \item From (U3) and (1), we have
               \begin{center}
               $\forall\neg \forall x=\forall(\forall x\rightarrow 0)=\forall x\rightarrow \forall 0=\forall x\rightarrow 0=\neg\forall x$.
               \end{center}
               \item From (U2), we have
               \begin{eqnarray*}
               (\forall x\wedge \forall y)\rightarrow \forall(\forall x\wedge \forall y)&=& (\forall x\rightarrow \forall(\forall x\wedge \forall y))\vee (\forall y\rightarrow \forall(\forall x\wedge \forall y))\\ &=& \forall(\forall x\rightarrow (\forall x\wedge \forall y)\vee \forall(\forall y\rightarrow (\forall x\wedge \forall y)\\ &=& \forall(\forall x\rightarrow \forall y)\vee \forall(\forall y\rightarrow \forall x)\\ &=&(\forall x\rightarrow \forall y)\vee (\forall y\rightarrow \forall x)\\ &=&1,
               \end{eqnarray*}
               which implies $(\forall x\wedge \forall y)\leq \forall(\forall x\wedge \forall y)$. Also, further by (U1) and (4), we obtain
               \begin{center}
               $\forall x\wedge \forall y\leq \forall(\forall x\wedge \forall y)\leq \forall(x\wedge y)\leq \forall x\wedge \forall y$.
                \end{center}
               \item From $x\odot y\leq x\odot y$, we get $y\leq x\rightarrow(x\odot y)$. Applying (4),(5), we get $\forall y\leq \forall x\rightarrow \forall(x\odot y)$. Thus, $\forall(x)\odot \forall(y)\leq \forall(x\odot y)$ by Definition 2.1(3).
               \item It follows from (7) and (8).
               \item Let $y\in \forall L$. Then there exists $x\in L$ such that $y=\forall x$. Hence $\forall y=\forall\forall x=\forall x=y$. It follows that $y\in L_{\forall}$. Conversely, if $y\in L_{\forall}$, we have $y\in\forall L$. Therefore, $\forall L=L_{\forall}$.
               \item  (7) and (12) imply that $\rightarrow$ and $\odot$ are preserved, respectively. (1) and (2) imply that $0,1\in \forall L$. (3) implies that $\forall$ is preserved. Thus, $\forall L$ is a subalgebra of $L$.
             \end{enumerate}
\end{proof}

In what follows, we focus our study on two main subvarieties of $\mathbb{UMTL}$: MV-algebras with universal quantifiers and monadic Boolean algebras. \medskip

 An algebra $(A,\wedge,\vee,\rightarrow,\odot,\forall,0,1)$ is said to be MV-algebra with a universal quantifier (UMV-algebra for short) if $(A,\wedge,\vee,\rightarrow,\odot,0,1)$ is an MV-algebra and in addition $\forall$ satisfies the following identities:

\begin{enumerate}[$(\forall 1)$]
  \item $\forall 1=1$,
  \item $\forall x\leq x$,
  \item $\forall(x\vee \forall y)=\forall x\vee \forall y$,
  \item $\forall(x\rightarrow y)\rightarrow (\forall x\rightarrow \forall y)=1$,
  \item $\forall(\forall x\rightarrow \forall y)=\forall x\rightarrow \forall y$.
\end{enumerate}
The variety of UMV-algebras is denoted by $\mathbb{UMV}$.

\begin{theorem} \emph{The subvariety of $\mathbb{UMTL}$ determined by the equation
\begin{center}
(MV) $(x\rightarrow y)\rightarrow y=(y\rightarrow x)\rightarrow x$
 \end{center}
 is term-equivalent
to the variety $\mathbb{UMV}$.}
\end{theorem}
\begin{proof}
Let $(L,\wedge,\vee,\rightarrow,\odot,0,1,\forall)$ be a UMTL-algebra that satisfies the MV-condition. Now, we prove that is a  UMV-algebra. Indeed, $(\forall 1)$,$(\forall 2)$,$(\forall 3)$,$(\forall 4)$ and $(\forall 5)$ are precisely Proposition 3.4(2), (U1), (U2), Propositions 3.4(5) and (7), respectively.  Thus $(L,\wedge,\vee,\rightarrow,\odot,0,1,\forall)$ is a UMV-algebra.

 Conversely, let $(L,\wedge,\vee,\rightarrow,\odot,0,1,\forall)$ be a UMV-algebra. Now, we  prove that is a UMTL-algebra satisfies the MV-condition. Indeed, (U1) and (U2) are precisely $(\forall2)$ and $(\forall3)$, respectively. In order to show (U3), by Proposition 3.4(3) and (6), we have
 \begin{center}
 $\forall(\forall x\rightarrow y)\leq \forall\forall x\rightarrow \forall y=\forall x\rightarrow \forall y$.
  \end{center}
Moreover, by $(\forall 1)$, we have $\forall x\rightarrow \forall y\leq \forall x\rightarrow y$, further by Proposition 3.4(5), we have
 \begin{center}
 $\forall x\rightarrow \forall y\leq \forall(\forall x\rightarrow y)$.
  \end{center}
 So (U3) holds.

 Thus $(L,\wedge,\vee,\rightarrow,\odot,0,1,\forall)$ is a UMTL-algebra satisfies the MV-condition.

\end{proof}

 An algebra $(L,\wedge,\vee,\neg,\exists,0,1)$ is said to be monadic Boolean algebra if $(L,\wedge,\vee,\neg,0,1)$ is a Boolean algebra and in addition $\exists$ satisfies the following identities:

\begin{enumerate}[$(\exists 1)$]
  \item $\exists 0=0$,
  \item $x\leq \exists x$,
  \item $\exists(x\wedge \exists y)=\exists x\wedge\exists y$.
\end{enumerate}
The variety of monadic Boolean algebras is denoted by $\mathbb{MBA}$.

\begin{theorem} \emph{The subvariety of $\mathbb{UMTL}$ determined by the equation
\begin{center}
(EM) $x\vee \neg x=1$
\end{center}
 is term-equivalent
to the variety $\mathbb{MBA}$.}
\end{theorem}
\begin{proof}
Let $(L,\wedge,\vee,\rightarrow,\odot,0,1,\forall)$ be a UMTL-algebra that satisfies the EM-condition. Then $\exists x=\neg\forall\neg x$ for any $x\in L$. Now, we prove that is a  monadic Boolean algebra.  Indeed, $(\exists 1)$,$(\exists 2)$ and $(\exists 3)$ are precisely dual to Proposition 3.4(2), (U1) and (U2), respectively.  Thus $(L,\wedge,\vee,\rightarrow,\odot,0,1,\forall)$ is a monadic Boolean algebra.

 Conversely, let $(L,\wedge,\vee,\neg,\exists,0,1)$ be a monadic Boolean algebra. Then $(L,\wedge,\vee,\odot,\\\rightarrow,\forall,0,1)$ is a UMV-algebra, where $x\odot y=x\wedge y$, $x\rightarrow y=\neg(x\odot \neg y)$. The rest of proof is similar to that of Theorem 3.5. Thus $(L,\wedge,\vee,\rightarrow,\odot,0,1,\forall)$ is a UMTL-algebra satisfies the EM-condition.
\end{proof}
\begin{remark} \emph{Theorem 3.5 and 3.6 show that the MTL-algebras with universal quantifiers essentially natural generalize MV-algebras with universal quantifiers and monadic Boolean algebras. }
\end{remark}

\section{Representations of UMTL-algebras}

In this section, we characterize classes of UMTL-algebras, likeness representable, strong, simple and semisimple UMTL-algebras and give some representations of them.

\begin{definition}\emph{ A filter $F$ of $L$ is called a \emph{U-filter} of $(L,\forall)$ if it verifies
\begin{center}
 $x\in F$ implies $\forall x\in F$.
 \end{center}}
\end{definition}

Let $(L,\forall)$ be a UMTL-algebra. For any nonempty subset $X$ of $L$, we denote by $\langle X \rangle_\forall$ the U-filter of $(L,\forall)$ generated by $X$, that is, $\langle X \rangle_\forall$ is the smallest monadic filter of $(L,\forall)$ containing $X$. Indeed,
\begin{center}
$\langle X\rangle_\forall=\{x\in L|x\geq \forall x_1\odot\forall x_2\odot\cdots \odot\forall x_n, x_i\in X, n\geq 1\}$
 \end{center}
 and
 \begin{center}
 $\langle a\rangle_\forall=\{x\in L|x\geq (\forall a)^n, n\geq 1\}$.
 \end{center}
  Also,
If $F$ is a U-filter of $(L,\forall)$ and $x\notin F$, then we put
\begin{center}
$\langle F,x\rangle_\forall:=\langle F\cup \{x\}\rangle_\forall=\{y\in L|x\geq f\odot (\forall x)^n, f\in F\}=F\vee [\forall x)$.
\end{center}
The set of all U-filter of $(L,\forall)$, which will represent by $UF[L,\forall]$, is an algebraic closure system and is a subset of the set of all lattice filters of $(L,\wedge,\vee,0,1)$. Therefore, $UF[L,\forall]$
is an algebraic lattice in which meet is the set-intersection and the join is defined as follows: if $F_1,F_2\in UF[L,\forall]$, then
\begin{center}$F_1\vee F_2=\langle F_1\cup F_2\rangle_\forall=\{x|x\geq f_1\odot f_2, f_1\in F_1,f_2\in F_2\}$.
\end{center}

\begin{example} \emph{Let $(L,\forall)$ be the UMTL-algebra in Example 3.2. Then $\{1\}$, $\{1,d\}$, $\{1,b,c\}$ and $L$ are U-filters of
$(L,\forall)$.}
\end{example}

The following example indicates that the concept of U-filters in UMTL-algebras is not the same as that of filters in MTL-algebras.

\begin{example}\emph{ Let $L=[0,1]$ be a unit interval. Define $\wedge,\vee,\odot,\rightarrow$ as follows: $x\wedge y=\min\{x,y\}$, $x\vee y=\max\{x,y\}$,
\begin{center}
$x\odot y=
\begin{cases}
0, & x\leq \neg y \\
x\wedge y, & x> \neg y
\end{cases}$
{\qquad}
$x\rightarrow y=
\begin{cases}
1, & x\leq y \\
\neg x\vee y, & x> y.
\end{cases}$
\end{center}
 Now we define $\forall$ as follow: for any $x\in L$,
\begin{center}
$\forall x=
\begin{cases}
1, & x=1 \\
0, & x\neq 1.
\end{cases}$
\end{center}
Then $(L,\forall)$ is a UMTL-algebra. Moreover, it is easy to check that $(\frac{1}{2},1]$ is a filter of $L$ but not
a U-filter of $(L,\forall)$.}
\end{example}

 There exists a correspondence between the set of U-congruences and U-filters.

\begin{theorem} \emph{Let $(L,\forall)$ be a UMTL-algebra. Then the lattice of U-congruences is isomorphic to the set of U-filters. Indeed, let
\begin{center}
$f:UC[L,\forall]\rightarrow UF[L,\forall]$
 \end{center}
 be defined by: if $\equiv$ is a U-congruence, then $f(\equiv)$ is the U-filter $F_\equiv=\{a\in L|a\equiv 1\}$. Also, the function $f$ is an isomorphism such that if $F$ is a U-filter, then $f^{-1}(F)$ is a U-congruence $\equiv_F$ defined by $a\equiv_F b$ if and only if $a\rightarrow b,b\rightarrow a\in F$.  }
\end{theorem}

As a direct consequence, we have the following fact.

\begin{proposition}\emph{Let $(L,\forall)$ be a UMTL-algebra and $F$ be a U-filter of $(L,\forall)$. Then $(L/F,\forall_F)$ is a UMTL-algebra,
where
\begin{center}
$\forall_F([x])=[\forall x]$
\end{center}
for any $x\in L$.}
\end{proposition}

Now, we give some characterizations of representable UMTL-algebras.

\begin{definition}\emph{A UMTL-algebra is called \emph{representable} if it is a subdirect product of a system of linearly ordered UMTL-algebras.}
\end{definition}

\begin{theorem}\emph{Let $(L,\forall)$ be a UMTL-algebra. Then the following statements are equivalent: for any $x,y\in L$,}
\begin{enumerate}[(1)]
  \item\emph{$(L,\forall)$ is representable,}
  \item \emph{$\forall(x\rightarrow y)\vee(y\rightarrow x)=1$,}
  \item \emph{$x\vee y=1$ implies $x\vee \forall y=1$,}
  \item \emph{any minimal prime filter is a U-filter of $(L,\forall)$.}
\end{enumerate}
\end{theorem}
\begin{proof} $(1)\Rightarrow (2)$ If $(L,\forall)$ is representable, then an equation holds in a representable UMTL-algebra if and only if it holds in the linearly ordered UMTL-algebras. Thus, we only need to prove that
\begin{center}
$\forall(x\rightarrow y)\vee (y\rightarrow x)=1$
 \end{center}
holds in any linearly ordered UMTL-algebra. In fact, if~$x\leq y$, then~$\forall(x\rightarrow y)=1$, and hence $\forall(x\rightarrow y)\vee (y\rightarrow x)=1$. Conversely, if~$y\leq x$,~then~$y\rightarrow x=1$, and hence~$\forall(x\rightarrow y)\vee (y\rightarrow x)=1$.

$(2)\Rightarrow (3)$~If~$x\vee y=1$, then~$x\rightarrow y=y$,~$y\rightarrow x=x$, and hence
\begin{center}
$x\vee \forall y=\forall(x\rightarrow y)\vee (y\rightarrow x)=1$,
\end{center}
 which implies that (3) holds in any UMTL-algebras.

$(3)\Rightarrow (4)$ If $F$ is a prime filter of $L$, and $x,y\in F$, then there exists~$z\in L$~such that~$z\notin F$ and $z\wedge x=1$. Since $F$~is prime filter and $z\notin F$, by $z\vee \forall x=1\in F$, we have $\forall x\in F$. Similarly, we have~$\forall y\in F$, and hence $\forall(x\odot y)\in F$, which implies that $F$ is a U-filter of $(L,\forall)$.

$(4)\Rightarrow(1)$ Let~$(L,\forall)$~be a UMTL-algebra and ~$\mathcal{F}$ be the set of all the minimal prime filters of MTL-algebra~$L$. Notice that any MTL-algebra~$L$~is a subdirect product of the family~$\{L/\sim_F|F\in \mathcal{F}\}$,~and~let
\begin{center}
$\imath:L\rightarrow \Pi_{F\in \mathcal{F}}L/F$
\end{center}
be the corresponding representation. Then follows from Proposition 4.5 that $(L/F,\forall_F)$~is a UMTL-algebra. It is straightforward that~$\imath$~is a presentation of $(L,\forall)$ as a subdirect product of the family $\{L/\sim_F|F\in \mathcal{F}\}$.
\end{proof}

The next theorem shows that any linearly ordered MTL-algebra has a structure of representable UMTL-algebra.

\begin{theorem}\emph{Let $L$ be an MTL-algebra. Then the following statements are equivalent:}
\begin{enumerate}[(1)]
\item\emph{$(L,\Delta)$ is a representable UMTL-algebra, where $\triangle$ is defined in Example 3.6.}
\item\emph{$L$ is a linearly ordered MTL-algebra.}
\end{enumerate}
\end{theorem}
\begin{proof} $(1)\Rightarrow(2)$ Let $(L,\Delta)$ be a representable UMTL-algebra and $x,y$ be two arbitrary elements of $L$. If $x\nleq y$, then $x\rightarrow y\neq 1$. So $\Delta(x\rightarrow y,1)=0$, further by Theorem 4.7(2), we get $y\rightarrow x=1$ and hence $y\leq x$. Thus, $L$ is a linearly ordered MTL-algebra.

$(2)\Rightarrow (1)$ Let $L$ be a linearly ordered MTL-algebra and $x,y$ be two arbitrary elements of $L$. If $x\leq y$ then $x\rightarrow y=1$ and, if $y\leq x$ then $y\rightarrow x=1$, hence in both cases, $\Delta(x\rightarrow y)\vee (y\rightarrow x)=1$. Thus, $(L,\Delta)$ is a representable UMTL-algebra.
\end{proof}

As an application of Theorem 4.8, we give some remarks as follows.

\begin{remark}
 \begin{enumerate}[(a)]
 \item\emph{If we replace Theorem 4.8(2) with $L$ is representable, then it is not true in general. Indeed, there exists a UMTL-algebra in Example 3.2 which is not representable, where $L$ is representable but is not linear.}
 \item\emph{Every representable MTL-algebra can be embedded in a representable UMTL-algebra. Indeed, if $L$ is representable, then $L$ is isomorphic to a subdirect product of linearly ordered MTL-algebras. From Theorem 4.8, any linearly ordered MTL-algebra has a structure of representable UMTL-algebra. Moreover, the class of representable UMTL-algebras is also a variety, so a direct product of representable UMTL-algebra is still a representable UMTL-algebra.}

     \end{enumerate}
   \end{remark}

 Serval authors introduced the strong universal quantifier, which is a universal quantifier $\forall$ satisfies the following condition:
\begin{center}
 $(\ast)$  $\forall(x\vee y)=\forall x\vee \forall y$.
 \end{center}
and proved that every strong algebra is representable \cite{Wang}. Indeed, strong and representable UMTL-algebras coincide, see the following theorem.

\begin{theorem}\emph{Let $(L,\forall)$ be a UMTL-algebra. Then the following statements are equivalent:}
\begin{enumerate}[(1)]
  \item\emph{$(L,\forall)$ is representable,}
  \item \emph{$(L,\forall)$ is strong.}
\end{enumerate}
\end{theorem}
\begin{proof} $(1)\Rightarrow (2)$ If $(L,\forall)$ is representable, then an equation holds in a general UMTL-algebra if and only if it holds in the linearly ordered UMTL-algebras. Hence
\begin{center}
$\forall(x\vee y)=\forall(x)\vee \forall(y)$
\end{center}
holds in all UMTL-algebras, which implies that $(L,\forall)$ is a strong UMTL-algebra.

$(2)\Rightarrow (1)$ If $(L,\forall)$ is a strong UMTL-algebra, then
\begin{center}
$1=\forall(x\rightarrow y)\vee \forall(y\rightarrow x)\leq \forall(x\rightarrow y)\vee (y\rightarrow x)$,
 \end{center}
 and hence
 \begin{center}
 $\forall(x\rightarrow y)\vee (y\rightarrow x)=1$,
 \end{center}
which implies that $(L,\forall)$ is representable follows from Theorem 4.7(2).
\end{proof}

\begin{remark}\emph{ $(1)$ Theorem 4.10 shows that strong UMTL-algebras are not a new class of UMTL-algebras but coincide with representable UMTL-algebras.}

\emph{ $(2)$ It is naturally verified that every strong universal quantifier $\forall$ on Boolean algebra is equivalent to the identity. Indeed, every strong monadic Boolean algebra is representable as a 2-element monadic Boolean algebra, in which every universal quantifier is equivalent to be the identity.}
\end{remark}

\begin{theorem}\emph{ Let $(L,\forall)$ be a UMTL-algebra and $F$ a proper U-filter of $(L,\forall)$. Then the following statements are equivalent:}
\begin{enumerate}[(1)]
\item \emph{$F$ is a maximal U-filter of $(L,\forall)$,
  \item for any $a\notin F$, there is an integer $n\geq 1$ such that $((\forall a)^n)^\ast\in F$.}
\end{enumerate}
\end{theorem}
\begin{proof} $(1)\Rightarrow (2)$ Let $F$ be a maximal U-filter of $(L,\forall)$ and $a\notin F$. Then $\langle F,a\rangle_\forall=L$, which implies $0\in \langle F,a\rangle_\forall$. Hence there is $f\in F$ and an integer $n\geq 1$ such that $0=f\odot (\forall a)^n$, that is, $\forall 0=0\geq \forall f\odot (\forall a)^n$. Thus, $\forall f\leq ((\forall a)^n)^\ast$. Therefore, $((\forall a)^n)^\ast\in F$.

$(2)\Rightarrow (1)$ If $a$ satisfy the condition (2), then $((\forall a)^n)^\ast\odot ((\forall a)^n)=0$, and hence $((\forall a)^n)^\ast\in F$. Hence $0\in \langle F,a\rangle_\forall$, that is, $\langle F,a\rangle_\forall=L$. Therefore, $F$ is a maximal U-filter of $(L,\forall)$.
\end{proof}

\begin{definition}\emph{A UMTL-algebra $(L,\forall)$ is said to be \emph{simple} if it has exactly two U-filters: $\{1\}$ and $L$.}
\end{definition}

\begin{theorem} \emph{Let $(L,\forall)$ be a UMTL-algebra. Then the following statements are equivalent:}
\begin{enumerate}[(1)]
  \item \emph{$(L,\forall)$ is simple,
  \item $\forall L$ is simple,
  \item $L_{\forall}=\{0,1\}$,
  \item $\langle 1\rangle$ is the only proper U-filter in $(L,\forall)$,
  \item for any $x\in L$, $x\neq 1$ implies $ord(\forall x)<\infty$.}
\end{enumerate}
\end{theorem}
\begin{proof} $(1)\Rightarrow(2)$ Let $(L,\forall)$ be simple and $F$ be a filter of $\forall L$ and $F\neq\{1\}$.  Then we will prove that $\forall L$ is simple. Considering the set
\begin{center}
$F_f=\{z\in L|z\geq f$ for a certain $f\in F\}$.
 \end{center}
 If $x,y\in F_f$, then there exist $f_1, f_2\in F$ such that $x\geq f_1, y\geq f_2$, and hence $x\odot y\geq f_1\odot f_2\in F$, which implies $x\odot y\in F_f$.
If $x\in F_f$ and $x\leq y$, then $y\in F_f$. Moreover, if $x\in F_f$, then $x\geq f, f\in F$, and hence $\forall x\geq \forall f=f$ (since $f\in \forall L$), which implies $\forall x\in F_f$. Thus, $F_f$ is a U-filter of $(L,\forall)$. Since $(L,\forall)$ is simple, and $F_f\neq\{1\}$ (since $F\subseteq F_f$). It follows that $F_f=L$, and so $0\in F_f$, hence $F=\forall L$, that is, $\forall L$ is simple.

$(2)\Rightarrow(1)$ Let $F$ be a U-filter of $(L,\forall)$. Then $F\cap \forall L$ is a filter of $\forall L$, and hence $F\cap \forall L=\{1\}$ or $F\cap \forall L=\forall L$. If $F\cap \forall L=\forall L$, then $\forall L\subseteq F$. Since $0\in\forall L$, we have $F=L$. If $F\cap \forall L=\{1\}$ and $x\in F$, then $\forall x\in F\cap\forall L$. So $\forall x=1$, that is, $x=1$ (since Ker$(\forall)=\{1\}$), and hence $F=\{1\}$. Thus, $(L,\forall)$ is simple.

$(2)\Leftrightarrow(3)$ and $(1)\Leftrightarrow(4)$ are follows from Definition 4.11.

$(1)\Leftrightarrow(5)$ From $(1)\Leftrightarrow (4)$, we obtain that $(L,\forall)$ is simple if and only if $\{1\}$ is the unique proper U-filter of $(L,\forall)$. Also, $\{1\}$ is the unique proper U-filter of $(L,\forall)$ if and only if for any $x\in L$, if $x\neq 1$, then $\langle x\rangle_\forall=L$ if and only if $0\in \langle x\rangle_\forall$ if only if $(\forall x)^2=0$ if and only if  for any $x\in L$, $x\neq 1$ implies $ord(\forall x)<\infty$.
\end{proof}

\begin{theorem} \emph{Let $(L,\forall)$ be a UMTL-algebra and $F$ be a proper U-filter of $(L,\forall)$. Then the following statements are equivalent:}
\begin{enumerate}[(1)]
 \item \emph{$(L/F,\forall_F)$ is a simple UMTL-algebra,
\item  $F$ is the maximal U-filter of $(L,\forall)$.}
\end{enumerate}
\end{theorem}
\begin{proof} $(1)\Rightarrow (2)$ Let $G$ be a U-filter of $(L,\forall)$ and $F\subseteq G$. Taking $x\in G\backslash F$, we have $[x]_F\neq [1]_F$. Since $(L/F,\forall_F)$ is simple, we obtain $ord((\forall_F[x]_F)^n)<\infty$. It follows that there exists $n\in N$ such that $(\forall_F[x]_F)^n=[0]_F$, that is, $[(\forall(x))^n]_F=[0]_F$. Hence $((\forall(x))^n)^\ast\in F\varsubsetneq G$. Combining $(\forall(x))^n\in G$, we have $0\in G$. Hence $G=L$. Therefore, $F$ is the maximal U-filter of $(L,\forall)$.

$(2)\Rightarrow (1)$ Let $F$ be the maximal U-filter of $(L,\forall)$ and $[x]_F\neq [1]_F$. Then $x\notin F$. It follows that $\langle F,x\rangle_\forall=L$, which implies that there exists $f\in F$ and $n\in N$ such that $f\odot(\forall x)^n=0$. Hence $f\leq ((\forall x)^n)^\ast$. Thus $((\forall x)^n)^\ast\in F$. It follows that $[(\forall x)^n]_F=[0]_F$. This means $[(\forall x)]^n{_F}=[0]_F$, that is, $ord((\forall_F[x]_F)^n)<\infty$. Therefore, by Theorem 4.14, we obtain that $(L/F,\forall_F)$ is a simple UMTL-algebra.
\end{proof}

The intersection all maximal U-filters of $(L,\forall)$ is called the \emph{radical} of $(L,\forall)$ and is denoted by $URad(L,\forall)$.

\begin{definition}\emph{ A UMTL-algebra $(L,\forall)$ is said to be \emph{semisimple} if the intersection of all maximal U-congruences of $(L,\forall)$ is the U-congruence $[0]$.}
\end{definition}

Notice that in any UMTL-algebra $(L,\forall)$, the U-congruences are in bijective correspondence with the U-filters. Then follows that $(L,\forall)$ is semisimple if and only if $URad(L,\forall)=\{1\}$.

\begin{definition}\emph{ Let $(L_1,\forall_1)$ and $(L_2,\forall_2)$ be two UMTL-algebras. A homomorphism
 $f:L_1\rightarrow L_2$
  of MTL-algebras is called a U-homomorphism between $(L_1,\forall_1)$ and $(L_2,\forall_2)$ if its satisfies
  \begin{center}
  $f(\forall_1x)=\forall_2f(x)$
  \end{center}
  for any $x\in L_1$.}
 \end{definition}

\begin{theorem}\emph{Let $(L,\forall)$ be a UMTL-algebra. Then the following statements are equivalent:}
\begin{enumerate}[(1)]
 \item \emph{$(L,\forall)$ is semisimple,
\item  $(L,\forall)$ is a subdirect product of a family of simple UMTL-algebras.}
\end{enumerate}
\end{theorem}
\begin{proof} $(1)\Rightarrow (2)$ Let $(L,\forall)$ be semisimple. Then for any $x\in L$, $x\neq 1$, there exists a maximal U-filter $F\notin UMax[L,\forall]$ such that $x\notin F$. Thus, we can check that the map
\begin{center}
$\varphi:(L,\forall)\rightarrow (\prod_{F\in UMax[L,\forall]}L/F, \prod_{UMax[L,\forall]}\forall_F)$
\end{center}
given by
\begin{center}
$\varphi(x)=([x]_F,[y]_F)_{F\in UMax[L,\forall]}$
\end{center}
is an injective U-homomorphism and $\pi_F\circ \varphi$ is  a surjective U-homomorphism, where
\begin{center}
$\pi_F:(\prod_{G\in UMax[L,\forall]}L/G, \prod_{G\in UMax[L,\forall]}\forall_G)\rightarrow (L/F,\forall_F)$
\end{center}
 is the projection. Also, by Theorem 4.14, we know that $(L/F,\forall_F)$ is simple. Thus, $(L,\forall)$ is a subdirect product of simple UMTL-algebras $\{(L/F,\forall_F)\}_{F\in MMax[L,\forall]}$.

$(2)\Rightarrow (1)$ Let
\begin{center}
$\varphi:(L,\forall)\rightarrow (\prod_{i\in I}L_i,\prod_{i\in I}\forall_i,)$
\end{center}
be an injective U-homomorphism, where $(L_i,\forall_i,)(i\in I)$ are simple UMTL-algebras, and let
\begin{center}
$\pi_i\circ\varphi:(L,\forall,)\rightarrow (L_i,\forall_i)$
\end{center}
be a surjective U-homomorphism. Set
\begin{center}
Ker$(\pi_i\circ\varphi)=F_i$
 \end{center}
 for any $i\in I$. Then we can prove that $F_i(i\in I)$ is the maximal U-filter of $(L,\forall)$. Now, let $x\in \cap\{F_i|i\in I\}$. Then
 \begin{center}
 $\pi_i\circ \varphi(x)=1$ for all $i\in I$,
  \end{center}
 and hence $\varphi(x)=1$. Since $\varphi$ is injective, we obtain $x=1$. Hence
 \begin{center}
 $\cap\{F_i|i\in I\}=\{1\}$,
 \end{center}
 which implies that $(L,\forall)$ is semisimple.
\end{proof}

\section{The modal logic of MTL-algebras with universal quantifiers}

In this section, we establish modal monoidal t-norm based propositional logics and prove that is completeness with respect to the variety of MTL-algebras with universal quantifiers, and obtain that a necessary and sufficient condition for the modal monoidal t-norm based logic to be semilinear.\medskip

Adapting for the propositional case the axiomatization of MTL-algebras with universal quantifiers defined by Definition 3.1, we can define modal monoidal t-norm based logics $\mathbf{MMTL}$ as a logic which contains monoidal t-norm based logics $\mathbf{MTL}$, the formulas as the axioms schemes:

(M1) $\square \alpha\Rightarrow \alpha$,

(M2) $\square(\alpha\Rightarrow \square\beta \Rightarrow \square\beta)\equiv (\square \alpha\Rightarrow \square \beta)\Rightarrow \square \beta$,

(M3) $\square(\square\alpha\Rightarrow \beta)\equiv \square \alpha\Rightarrow \square \beta$.\\
and closed under Modus Ponens $\mathbf{MP}$: $\alpha$, $\alpha\Rightarrow \beta\vdash \beta$ and Necessitation Rule $\mathbf{Nec}: \alpha/\square \alpha$.\medskip

Now, we remind some well known notions that are used. Let $T$ be a \emph{theory}, that is a set of formulas in $\mathbf{MMTL}$. A formula is a \emph{theorem} if there exists a natural number $n\geq 1$ and a sequence of formulas $\alpha_1,\cdots, \alpha_n=\alpha$ such that, for any $i\in [n]$, one of the following conditions holds:

 $(1)$ $\alpha_i$ is an axiom;

 $(2)$ $\alpha_i\in T$;

 $(3)$ there are $j,k\leq i$ such that $\alpha_j$ is $\alpha_k\Rightarrow \alpha_i$;

 $(4)$ there exists $j\leq i$ such that $\varphi_i$ is $ \square\alpha_j$. \\
 The sequence $\alpha_1,\cdots,\alpha_n=\alpha$ is a \emph{proof} for $\alpha$. In fact, a formula is a theorem will be simply denoted by $T\vdash \alpha$. A formula $\alpha$ will be called a \emph{theorem} if it is provable from the empty set. This will be denote by $\vdash \alpha$. In this case, a proof for $\alpha$ will be a sequence of formulas $\alpha_1,\cdots, \alpha_n=\alpha$ such that for any $i\in \{1,\cdots,n\}$, one of the above conditions (1)-(4) is satisfied.

 A theory $T$ is said to be linear when for any two formulas $\alpha$ and $\beta$, either
 \begin{center}
 $ \alpha\Rightarrow \beta\in T$ or $\beta\Rightarrow  \alpha\in T$.
 \end{center}

  Also, $T$ is prime when for any two formulas $ \alpha$ and $\beta$ such that
  \begin{center}
  $ \alpha\sqcup \beta\in T$, either $ \alpha\in T$ or $\beta\in T$.
  \end{center}

  The fact that $\mathbf{MTL}$ is semilinear, since it guarantees that
\begin{center}
$\vdash ( \alpha\Rightarrow\beta)\sqcup (\beta\Rightarrow \alpha)$

$ \alpha\sqcup\beta, \alpha\Rightarrow\beta\vdash \beta$ and $ \alpha\sqcup\beta,\beta\Rightarrow \alpha\vdash  \alpha$.
\end{center}
These three conditions imply that in $\mathbf{MTL}$ linear and prime theories coincide.\medskip

In order to show that $\mathbf{MMTL}$ is complete, we apply a general result from Abstract Algebraic Logic ($\mathbf{AAL}$ shortly). We start from by showing that $\mathbf{MMTL}$ is an implicative logic (in the sense of Rasiowa)\cite{Rasiowa}, which is a logic if there is a binary (either primitive or definable by a formula) connective $\Rightarrow$ of its language such that the following hold:

(R) $\vdash \alpha\Rightarrow\alpha$,

(MP) $\alpha,\alpha\Rightarrow \beta\vdash \beta$,

(T)  $\alpha\Rightarrow\beta$, $\beta\Rightarrow\gamma\vdash \alpha\Rightarrow \gamma$,

(Cong) $\alpha\Rightarrow\beta,\beta\Rightarrow\alpha\vdash c(\gamma_1,\cdots,\gamma_i,\alpha,\cdots,\gamma_n)\Rightarrow c(\gamma_1,\cdots,\gamma_i,\beta,\cdots,\gamma_n)$,

(W) $\alpha\vdash \beta\Rightarrow\alpha$.\medskip

Most of these axioms hold trivially for $\mathbf{MTL}$ and the following fact (Proposition 5.1(3)) shows that (Cong) is also satisfies for the new unary connectives $\square$.
\begin{proposition}\emph{The following formulas are provable in $\mathbf{MMTL}$:}
\begin{enumerate}[(1)]
  \item $\vdash \square \overline{1}$,
  \item $\vdash \square \alpha\Rightarrow \square\square\alpha$,
  \item $\alpha\equiv \beta\vdash \square \alpha\Rightarrow \square\beta$,
  \item $\square \alpha,\alpha\Rightarrow\beta\vdash \square \beta$,
  \item $\square\alpha,\square(\alpha\Rightarrow\beta)\vdash \square \beta$,
  \item $\square\alpha\Rightarrow \beta\vdash \square\alpha\Rightarrow \square \beta$,
  \item $\vdash \square(\alpha\Rightarrow\beta)\Rightarrow (\square \alpha\Rightarrow \square \beta)$.
\end{enumerate}
\end{proposition}
\begin{proof}

$(1)$  It follows directly from (M1) taking $\alpha=\overline{1}$.

$(2)$  Taking $\beta=\square\alpha$, we have $\square\overline{1}\equiv \square \alpha\Rightarrow \square\square\alpha$. Also, by (1) and $\mathbf{MP}$, we have $\vdash \square \alpha\Rightarrow \square \square\alpha$.

$(3)$ By (M1), we have $\vdash \square \alpha\Rightarrow \alpha$, hence by transitivity of implication, we obtain $\alpha\Rightarrow \beta\vdash \square \alpha\Rightarrow \beta$. Also, using $\mathbf{Nec}$, we have $\alpha\Rightarrow \beta\vdash \square(\square \alpha\Rightarrow \beta)$, and using (M3), we obtain $\alpha\Rightarrow \beta\vdash \square \alpha\Rightarrow \square \beta$. Similarity, we can prove $\beta\Rightarrow \alpha\vdash \square \alpha\Rightarrow \square \beta$.

$(4)$ It follows directly from $(3)$ and $\mathbf{MP}$.

$(5)$ It follows directly from $(1)$, $(4)$ and $\mathbf{MP}$.

$(6)$ It follows from (M1), $(4)$ and $\mathbf{MP}$.

$(7)$ It follows from (5).
\end{proof}

  Thus $\mathbf{MMTL}$ is an implicative logic and hence is algebraizable in the sense of Blok and Pigozzi [4]. This gives us immediately the completeness with respect to its associated variety $\mathbb{UMTL}$ of  UMTL-algebras.

\begin{theorem} \emph{Let $T$ be a theory and $\alpha$ be a formula over $\mathbf{MMTL}$. Then the following statements are equivalent:}
\begin{enumerate}[(1)]
  \item \emph{$T\vdash \alpha$,
  \item for each UMTL-algebra $(L,\forall)$ and for every model $e$ of $T$, $e(\alpha)=1$,
  \item $[\alpha]_T=[\overline{1}]_T$ in $\mathbf{MMTL}$.}
\end{enumerate}
\end{theorem}

In $\mathbf{MMTL}$, the usual form of the deduction theorem does not hold. Indeed,
\begin{center}
$\alpha\vdash \square \alpha$, but $\nvdash \alpha\Rightarrow \square\alpha$,
\end{center} see the following example.

\begin{example}\emph{Let $(L,\forall)$ be a UMTL-algebra in Example 4.3. Then for any evaluation $e$ in this algebra, if $e(\alpha)=1$, then $e(\square \alpha)=1$. But for $e(\alpha)=\frac{1}{2}$ we have $e(\square\alpha)=0$, and hence $e(\alpha\Rightarrow \square\alpha)=0$. }
\end{example}

Actually, $\mathbf{MMTL}$ enjoys the same form of deduction theorem holding for logics with the $\triangle$ in \cite{Hajek}.

\begin{theorem}\emph{$T,\alpha\vdash \beta$ if and only if $T\vdash \square\alpha\Rightarrow\beta$.}
\end{theorem}
\begin{proof}  We prove by induction on every formula $\alpha_i$ $(1\leq i\leq n)$ of the given derivation of $\beta$ from $T\cup\alpha$ that $T\vdash \square \alpha\Rightarrow \alpha_i$.

If $\alpha_i=\alpha$, then the result follows due to (M1). If $\alpha_i\in T$ or is an instance of an axiom, then the result follows using $\mathbf{MP}$ and the derivability of the schema $\alpha_i\Rightarrow (\square \alpha\Rightarrow \alpha_i)$.

If $\alpha_i$ comes by  application of $\mathbf{MP}$ on previous formulas in the derivation, then the result follows, because from $\square \alpha\Rightarrow \alpha_k$ and $\square \alpha\Rightarrow (\alpha_k\Rightarrow \alpha_i)$ we may derive $(\square\alpha\odot \square\alpha)\Rightarrow (\alpha_k\odot(\alpha_k\Rightarrow \alpha_i)$ and hence also $\square \alpha\Rightarrow\alpha_i$, using transitivity of $\Rightarrow$ applied to Proposition 5.1.(2) and $(\alpha_k\odot(\alpha_k\Rightarrow \alpha_i)\Rightarrow \alpha_i$.

If $\alpha_i=\square \alpha_k$  comes using $\mathbf{Nec}$ from $\alpha_k$, then from $\square \alpha\Rightarrow \alpha_k$, we may derive $\square \alpha\Rightarrow\square \alpha_k$ using Proposition 5.1(7).

Conversely, to the derivation given by the hypothesis add a step
with $\alpha$. In the next step put $\square\alpha$, which follows from the previous formula using $\mathbf{Nec}$. Finally, derive $\beta$ using $\mathbf{MP}$.

\end{proof}

It is well known that $\mathbf{MTL}$ is algebraizable and strongly complete with respect to the class of linearly ordered MTL-algebras \cite{Esteva}. However, unlike the case of $\mathbf{MTL}$, $\mathbf{MMTL}$ is not semilinear, that is, it is not complete with respect to the class of linearly ordered UMTL-algebras. The reason is that the  {\bf disjunction form of the rule $\square$},
\begin{center}
from $\alpha\sqcup \beta$ derive $\alpha\sqcup \square\beta$,
\end{center} is not derivable in $\mathbf{MMTL}$, see the following example.

\begin{example}\emph{Let $L$ be an MTL-algebra in Example 3.2. Now, we define $\forall$ as follows:}
\begin{center}
$\forall x=
\begin{cases}
1, & x=1 \\
b, & x=b,c\\
d, & x=d\\
0, & x=0,a
\end{cases}$
\end{center}
\emph{Then $(L,\forall)$ is a UMTL-algebra. Indeed, it is clear that $d\vee c=1$, while $d\vee \forall c=d\vee b=d\neq 1$.}
\end{example}

Then it remains the problem of axiomatizing the minimal semilinear extension of $\mathbf{MMTL}$, that is, we provide conditions under which the logic $\mathbf{MMTL_\ell}$ extending $\mathbf{MMTL}$ is semilinear. In fact, we have the following result regarding to axiomatization of the least semilinear extension related to $\mathbf{MMTL}$.

\begin{theorem} \emph{Let $\mathbf{MMTL_\ell}$  be an expansion of $\mathbf{MMTL}$ plus the disjunction form of the rule $\square$. Then $\mathbf{MMTL_\ell}$ is semilinear, i.e., $\mathbf{MMTL_\ell}$ is complete with respect to the class of linearly ordered UMTL-algebras.}
\end{theorem}
\begin{proof} This is a well known consequence of the Axiomatization of the least semilinear extension using the representable UMTL-algebra obtained as the free UMTL-algebras by the disjunction form of the rule $\square$.
\end{proof}

\section{Conclusions}

 Motivated by  previous research about quantifiers on algebras, we investigated MTL-algebras with universal quantifiers. In this paper, we study some properties of UMTL-algebras and discuss relations among UMTL-algebras, UMV-algebras and monadic Boolean algebras. Then we characterize classes of UMTL-algebras and give some representations of them. Finally, we establish monoidal t-norm based propositional logics and study the semilinearity of them. Since the above topics are of current interest, we suggest further directions of research:
\begin{enumerate}[(1)]
  \item Constructing topological spaces and giving some topological representations of UMTL-algebras.
  \item Focusing on varieties of UMTL-algebras. In particular, one can investigate locally finite, finitely approximated and splitting varieties of UMTL-algebras as well as varieties with the disjunction and  existence properties.
\end{enumerate}

\medskip
\noindent\textbf{Acknowledgments}
\medskip

This study was funded by a grant of National Natural Science Foundation of China (61976244,11961016,11901451), the Innovation Talent Promotion Plan of Shaanxi Province for Young Sci-Tech New Star (2017KJXX-60) and the Natural Science Basic Research Plan in Shaanxi Province of China (2019JQ-816) and Natural Science Foundation of Education Committee of Shannxi Province (19JK0653).


\end{document}